\documentclass[10pt,a4paper]{amsart}

\usepackage{amsmath, amsthm, amssymb}
\usepackage{color,url}
\usepackage[all]{xy}

\newenvironment{myenum}[1]
	{	\renewcommand{\theenumi}{\textup{(#1\arabic{enumi})}}
		
		\begin{enumerate}}
	{\end{enumerate}}

\newenvironment{myenump}[1]
	{	\renewcommand{\theenumi}{\textup{(#1\arabic{enumi}')}}
		
		\begin{enumerate}}
	{\end{enumerate}}
	
\DeclareMathOperator{\Aut}{Aut}
\DeclareMathOperator{\Inn}{Inn}
\DeclareMathOperator{\Out}{Out}
\DeclareMathOperator{\MCG}{MCG}
\DeclareMathOperator{\PGL}{PGL}

\DeclareMathOperator{\Centr}{Z}
\DeclareMathOperator{\Norm}{N}
\DeclareMathOperator{\Fix}{Fix}
\DeclareMathOperator{\Id}{Id}
\DeclareMathOperator{\CAT}{CAT}

\newcommand{\NN}{\mathbf{N}}
\newcommand{\ZZ}{\mathbf{Z}}
\newcommand{\RR}{\mathbf{R}}

\newcommand{\beep}{AS }
\newcommand{\Beep}{AS }
\newcommand{\BP}{AS}

\newtheorem{theorem}{Theorem}[section]
\newtheorem{alphtheorem}{Theorem}

\newtheorem{lemma}[theorem]{Lemma}
\newtheorem{prop}[theorem]{Proposition}

\newtheorem{corollary}[theorem]{Corollary}

\theoremstyle{remark}

\newtheorem*{observation*}{Observation}
\newtheorem{question}[theorem]{Question}
\newtheorem{remark}[theorem]{Remark}
\newtheorem{example}[theorem]{Example}

\title{Virtually splitting the map from $\Aut(G)$ to $\Out(G)$}
\author{Mathieu Carette}
\address{\tt Universit\'e catholique de Louvain, 
IRMP,
Chemin du Cyclotron 2, bte L7.01.01, 
1348 Louvain-la-Neuve,
Belgium}
\email{\tt mathieu.carette@uclouvain.be}
\thanks{The author is a Postdoctoral Researcher of the F.R.S.-FNRS (Belgium).}

\begin{document}
	
	\begin{abstract}
		We give an elementary criterion on a group $G$ for the map $\Aut(G) \to \Out(G)$ to split virtually. This criterion applies to many residually finite $\CAT(0)$ groups and hyperbolic groups, and in particular to all finitely generated Coxeter groups. As a consequence the outer automorphism group of any finitely generated Coxeter group is residually finite and virtually torsion-free.
	\end{abstract}
	\maketitle
	
	Given a group epimorphism $f : G \to H$ we say that $f$ \textbf{splits virtually} if there exist a finite index subgroup $H' < H$ and a homomomorphism $g : H' \to G$ with $f \circ g = \Id$. In other words, the exact sequence $\ker f \to f^{-1}(H') \to H'$ splits, so that $G$ is virtually a semidirect product $\ker f \rtimes H'$.
	
	Given a finitely generated residually finite group $G$,  we investigate when the quotient map $\Aut(G) \to \Out(G) = \Aut(G) / \Inn(G)$ splits virtually. This implies in particular that $\Out(G)$ is residually finite (see Theorem~\ref{theorem:Baumslag} below). Note that there are examples of groups with residually finite $\Out(G)$ but such that the map $\Aut(G) \to \Out(G)$ does not split virtually. Indeed, if $G = \pi_1(S_g)$ for $S_g$ a closed orientable surface of genus $g \geq 3$, then $\Out(G) = \MCG^\pm(S_g)$ is residually finite \cite{Grossman74} but Mess showed in \cite{Mess90} that no finite index subgroup of the Torelli subgroup of $\Out(G)$ lifts to a subgroup of $\Aut(G)$.
	
	We give an elementary criterion, namely the existence of an \emph{\beep subgroup} $H$ of $G$, which implies that the map $\Aut(G) \to \Out(G)$ splits virtually. The idea is to lift outer automorphisms to automorphisms that fix $H$ pointwise. As a particular case of our criterion we show 
	
	\begin{alphtheorem} \label{thm:main_sample} Let $G$ be a finitely generated residually finite group, and let $H$ be a finite subgroup of $G$ with finite centralizer in $G$. Suppose moreover that $G$ has finitely many conjugacy classes of finite subgroups. Then the map $\Aut(G) \to \Out(G)$ splits virtually. In particular $\Out(G)$ is residually finite.
	\end{alphtheorem}
	
	We further give a geometric version of this criterion applying to many $\CAT(0)$ groups with torsion.
	\begin{alphtheorem} \label{thm:geometricsplit}
		Let $G$ be a residually finite group acting properly discontinuously and cocompactly on a complete $\CAT(0)$ space $X$. Suppose there exists a finite subgroup $H < G$ such that the set of fixed points of $H$ is bounded. Then the map $\Aut(G) \to \Out(G)$ splits virtually. In particular $\Out(G)$ is residually finite.
	\end{alphtheorem}
	
	When $G$ is a finitely generated Coxeter group, we will show that one can take $H$ above to be a maximal finite subgroup. Using further the fact that $G$ is $\mathbf{R}$-linear, we obtain as a consequence the following
	\begin{alphtheorem} Let $G$ be a finitely generated Coxeter group. Then the map $\Aut(G) \to \Out(G)$ splits virtually. Moreover $\Out(G)$ is virtually residually $p$ for all but finitely many primes $p$. In particular $\Out(G)$ is residually finite and virtually torsion-free.  \label{thm:Coxeter}
	\end{alphtheorem}
	
	Recall that \textbf{Coxeter group} is a group $G$ defined by a presentation of the form \[\langle S \mid (st)^{m_{st}} = 1 \text{ for each } s,t\in S \text{ with } m_{st} \neq \infty \rangle\] where $m_{ss} = 1$ and $m_{st} = m_{ts} \in \NN_{\geq 2} \cup \{\infty\}$ for each $s \neq t \in S$. If $m_{st} \in \{2,\infty\}$ for $s \neq t$ then the Coxeter group is called \textbf{right-angled}.
	
	Splittings of the map $\Aut(G) \to \Out(G)$ were described by James \cite{James88} and Tits \cite{Tits88} for some right-angled Coxeter groups $G$. This was later extended to all right-angled Coxeter groups by M\"uhlherr \cite{Muhlherr98} and to graph products of finite cyclic groups by Gutierrez, Piggott and Ruane \cite{GutPigRua12}. More generally, for $G$ a graph product of cyclic groups, they give a description of the subgroup $\Aut'(G) < \Aut(G)$ which preserves the conjugacy class of each vertex group. This is then used to show that the map $\Aut'(G) \to \Out'(G)$ splits virtually.
	
	We emphasize here that the structure of the group $\Aut(G)$ is not well understood for $G$ a general finitely generated Coxeter group. In particular, it is not known in general whether $\Aut(G)$ is finitely generated, or exactly for which $G$ is $\Out(G)$ infinite. This is related to M\"uhlherr's twist conjecture \cite{Muhlherr06} which gives a conjectural description of all Coxeter generating sets of a Coxeter group. Particular cases of this conjecture have been settled, among which \emph{twist-rigid} Coxeter groups \cite{CapPrz10}, which consequently all have a finite outer automorphism group. 
	
	Also related is Question 2.1 in Bestvina's problem list which asks if a $\CAT(0)$ group has an infinite outer automorphism group if and only if it has an infinite order Dehn twist.	In the case of a Coxeter group $G$ an infinite order Dehn twist means a splitting as an amalgamated free product $G = A \ast_C B$ with $A \neq C \neq B$ and such that $\Centr_A(C)$ is infinite. We mention here that the answer to the above question is positive for all hyperbolic groups \cite{Levitt05, Carette11} (see also \cite{GuirLev12} for a statement on relatively hyperbolic groups). The difficulty in determining which hyperbolic groups $G$ have infinite $\Out(G)$ lies mostly in the presence of torsion (the torsion-free case being settled by a theorem of Paulin \cite{Paulin}). In contrast, our strategy for virtually splitting the map $\Aut(G) \to \Out(G)$ relies in many instances on the presence of suitable finite subgroups, but is hopeless for many torsion-free groups like orientable surface groups (as Mess' result above shows) or free groups.
	
	Ashot Minasyan pointed out the following application of virtually splitting the map $\Aut(G) \to \Out(G)$.
	\begin{remark}[Word problem for $\Out(G)$] Let $G$ be a finitely generated group with solvable word problem, and suppose moreover that $\Aut(G)$ is finitely generated. Then it is not hard to deduce a solution of the word problem for $\Aut(G)$ using that of $G$. Pursuing this strategy, one could try to solve the word problem for $\Out(G)$ using a solution to the multiple conjugacy problem in $G$. However, giving a virtual splitting of the map $\Aut(G) \to \Out(G)$ (if it exists) yields a simple solution to the word problem for $\Out(G)$ only relying on a solution of the word problem for $G$.
	\end{remark}

	The paper is organized as follows. Section~\ref{sec:residualfiniteness} is logically independent from the rest of the paper (apart from Theorem~\ref{theorem:Baumslag} and results from Section~\ref{sec:virtually_residually_p}) and serves as a motivation, describing other approaches to show that $\Out(G)$ are residually finite for certain groups $G$. In Section~\ref{sec:nicesubgroup} we describe an elementary strategy to show virtual splitting of the map $\Aut(G) \to \Out(G)$ (see Proposition~\ref{prop:strategy}), namely the existence of an \beep subgroup. We then describe general algebraic and geometric situations giving rise to \beep subgroups (see Propositions~\ref{prop:selfnorm} and~\ref{prop:CAT(0)}) proving Theorems~\ref{thm:main_sample} and~\ref{thm:geometricsplit}. In Section~\ref{sec:applications} we apply the strategy to various classes of discrete groups. We first apply Theorem~\ref{thm:geometricsplit} to all finitely generated Coxeter groups and to some Fuchsian groups with torsion. We then apply our strategy to (non-proper) actions on suitable trees. Finally, we show that Theorem \ref{thm:main_sample} applies to all one-relator groups with torsion.
	
	\section*{Acknowledgements}
	The author is grateful to Pierre-Emmanuel Caprace and Ashot Minasyan for helpful discussions.
	
	\section{Residual finiteness of outer automorphism groups} \label{sec:residualfiniteness}
	A group $G$ is \textbf{residually finite} if the intersection of its finite index subgroups is trivial. In this definition, one may without loss of generality restrict to normal subgroups, or even to characteristic subgroups if $G$ is finitely generated. Other equivalent ways to phrase residual finiteness are the following : 
	\begin{itemize}
		\item for any $1 \neq g\in G$ there is a finite quotient $\varphi: G \to F$ such that $\varphi(g) \neq 1$.
		\item for any finite $X \subset G$ there is a finite quotient $\varphi : G \to F$ such that $\varphi|_X$ is injective.
	\end{itemize}
	Moreover if $G$ is finitely generated, then the quotient $F$ above may be chosen characteristic (i.e. $\ker \varphi$ is $\Aut(G)$-invariant, so that there is an induced map $\varphi^* : \Aut(G) \to \Aut(F)$). This observation leads to the following classical result of Baumslag.
	\begin{theorem}[\cite{Baumslag63}] \label{theorem:Baumslag} Let $G$ be a finitely generated residually finite group. Then $\Aut(G)$ is residually finite.
	\end{theorem}
	
	In contrast with automorphism groups, residual finiteness of a finitely generated group $G$ is a priori unrelated to residual finiteness of $\Out(G)$. Indeed, Bumagin and Wise \cite{BumWise05} have shown that \emph{any} finitely presented group appears as $\Out(G)$ of some finitely generated \emph{residually finite} group $G$. In an opposite direction, Minasyan \cite{Minasyan09} showed that \emph{any} countable group can be realized as $\Out(G)$ of a finitely generated \emph{simple} group $G$.
	
	A natural strategy for establishing residual finiteness of $\Out(G)$ in the spirit of Baumslag's result was devised by Grossman. A group is $G$ is called \textbf{conjugacy separable} if given any two non-conjugate elements in $G$ there is a finite quotient in which the images of these elements are not conjugate. An automorphism $\alpha$ of a group $G$ is called \textbf{pointwise inner} if $\alpha(g)$ is conjugate to $g$ for all $g \in G$. Grossman \cite{Grossman74} showed that if $G$ is a finitely generated conjugacy separable group such that all pointwise inner automorphisms of $G$ are inner, then $\Out(G)$ is residually finite. She further successfully applied this criterion to $G = F_n$ and $G = \pi_1(S_g)$ for $S_g$ the orientable surface of genus $g$. Variants of this strategy (possibly replacing conjugacy separability with other properties) play a role to establish residual finiteness of $\Out(G)$ for many finitely generated groups, such as 
	\begin{itemize}
		\item $G$ a residually finite group with more than one end (Minasyan--Osin \cite{MinOsin10})
		\item $G$ a right-angled Artin group (Minasyan \cite{Minasyan12}, Charney--Vogtmann \cite{ChaVogt11})
		\item $G$ a one-ended hyperbolic group, or more generally $G$ a one-ended relatively hyperbolic group with residually finite parabolic subgroups which are not relatively hyperbolic (Levitt--Minasyan \cite{LevMin}).
		\item $G$ an even Coxeter group without subgroups of type $\widetilde{B_2}$ (Caprace--Minasyan \cite{CapMin12})
	\end{itemize}
	
	Note however that establishing conjugacy separability is a very hard problem in general, especially in the presence of torsion. Contrary to residual finiteness, if $H$ is a finite index subgroup of $G$, conjugacy separability $G$ does not imply conjugacy separability of $H$ \cite{ChagZal09, MartMin12}, nor does conjugacy separability of $G$ follow from that of $H$ \cite{Goryaga86}. Caprace and Minasyan were able to show conjugacy separability for the above Coxeter groups, but it remains an open question whether all finitely generated Coxeter groups are conjugacy separable. Our results provide an approach to proving residual finiteness of the outer automorphism group of a Coxeter group without invoking conjugacy separability. 
	
	\begin{remark}[Hyperbolic groups] Let $G$ be a hyperbolic group. From the results cited above \cite{MinOsin10, LevMin} it follows that if $G$ is either one-ended or residually finite, then $\Out(G)$ is residually finite\footnote{Combining these results with the fact that any group $G$ embeds in $\Out(G \ast \ZZ / 2 \ZZ)$ it follows that all hyperbolic groups have a residually finite outer automorphism group if and only if all hyperbolic groups are residually finite, which is a well-known open question.}. As we remarked in the introduction, Mess' result \cite{Mess90} shows that no such general statement can be made about virtually splitting the map $\Aut(G) \to \Out(G)$. This motivates partial results such those in Section~\ref{sec:applications}.
	\end{remark}
	
	\subsection{Virtually residually $p$ groups} \label{sec:virtually_residually_p}
		Let $p$ be a prime. A group $G$ is called \textbf{residually $p$} if for any $1 \neq g\in G$ there is a $p$-finite quotient $\varphi: G \to F$ such that $\varphi(g) \neq 1$. We now record results and observations for later use.
		\begin{remark} \label{rmk:torsion-free} Any torsion element in a residually $p$ group has order a power of $p$. It follows that if a group $G$ is virtually residually $p$ and virtually residually $q$ for distinct primes $p$ and $q$ then $G$ is virtually torsion-free. 
		\end{remark}
		The following $p$ version of Baumslag's result is due to Lubotzky.
		\begin{theorem}[{\cite[Proposition 2]{Lubotzky80}}] Let $G$ be a finitely generated virtually residually $p$ group. Then $\Aut(G)$ is virtually residually $p$.
		\end{theorem}
		Finally, we shall make use of linearity through a result of Platonov:
		\begin{theorem}[{\cite{Platonov68}}] Let $G$ be a finitely generated linear group over a field of characteristic $0$. Then $G$ is virtually residually $p$ for all but finitely many primes $p$.
		\end{theorem}
		
	\section{Algebraic and geometric criterions} \label{sec:nicesubgroup}
	
	Let $G$ be a group and $H$ be a subgroup. Let $f : \Aut(G) \to \Out(G)$ be the natural quotient map. We denote by
	\begin{itemize} 
	  \item $\Aut_{[H]}(G)$ the subgroup of $\Aut(G)$ which preserves the conjugacy class of $H$;
		\item $\Aut_H(G)$ the subgroup of $\Aut(G)$ which fixes $H$ setwise;
		\item $\Aut_{\underline H}(G)$ the subgroup of $\Aut(G)$ which fixes $H$ pointwise;
		\item $\Out_{[H]}(G) = f(\Aut_{[H]}(G)), \Out_{H}(G) = f(\Aut_{H}(G)), \Out_{\underline H}(G) = f(\Aut_{\underline H}(G))$
	\end{itemize}
	Remark that $\Out_{[H]}(G) = \Out_H(G)$. The group $\Out_{\underline H}(G)$ is sometimes called a \emph{McCool group} \cite{GuirLev12}, or the subgroup of $\Out(G)$ which acts \emph{trivially} on $H$.
	
	We say that $H$ is an \textbf{\beep subgroup} of $G$ (or $\Aut$-splitting subgroup) if the sequence 
	\begin{equation} \label{eqn:beepsubgroup} 1 \to \Aut_{\underline H}(G) \stackrel{f}{\to} \Out(G) \to 1 \tag{\BP}
	\end{equation} is virtually exact, i.e. if $\ker f|_{\Aut_{\underline H}(G)}$ is finite and $f(\Aut_{\underline H}(G)) = \Out_{\underline H}(G)$ has finite index in $\Out(G)$.
	
	\begin{prop} \label{prop:strategy} If $\Aut(G)$ is residually finite, and $H$ is an \beep subgroup of $G$, then the map $f: \Aut(G) \to \Out(G)$ splits virtually. In particular $\Out(G)$ is residually finite.
	\end{prop}
	\begin{proof}
		Since $\Aut(G)$ is residually finite, so is $\Aut_{\underline H}(G)$. Hence we can find a finite quotient $\psi : \Aut_{\underline H}(G) \to F$ such that $\psi$ is injective on the finite subgroup $\ker f|_{\Aut_{\underline H}(G)}$. Thus $\ker \psi$ is a finite index subgroup of $\Aut_{\underline H}(G)$, $f$ is injective on $\ker \psi$ and $f(\ker \psi)$ has finite index in $\Out(G)$. The desired virtual splitting is given by $O = f(\ker \psi)$ and $g = f^{-1}|_O$.
	\end{proof}
	
	\begin{example}[Trivial examples] Let $G$ be a group. Then $G$ is an \beep subgroup of $G$ if and only if $\Out(G)$ is finite, and $\{1\}$ is an \beep subgroup of $G$ if and only if $\Inn(G)$ is finite, i.e. $\Centr(G)$ has finite index in $G$. In the latter case, if $G$ is finitely generated, then it is residually finite, hence so is $\Aut(G)$ by Theorem~\ref{theorem:Baumslag}, so that $\Aut(G)$ is commensurable with $\Out(G)$.
	\end{example}
	
	\begin{remark}\label{remark:trivialcenter} If $G$ admits an \beep subgroup with trivial center, then the map $f|_{\Aut_{\underline H}(G)}$ is injective so that the map $f : \Aut(G) \to \Out(G)$ splits virtually. This does not imply in general that $\Out(G)$ is residually finite.
	\end{remark}
	
	We now restate the \beep subgroup condition.
	\begin{prop} \label{prop:restatement} Let $G$ be a group, and $H$ be a subgroup. Then $H$ is an \beep subgroup of $G$ if and only if the following three conditions hold:
	\begin{myenum}{\BP}
		\item $\Centr_G(H) / \Centr(G)$ is finite. \label{cond:finitecentralizer}
		\item A finite index subgroup of $\Aut(G)$ preserves the conjugacy class of $H$. \label{cond:preserve_conj_class}
		\item $\Out_{\underline H}(G)$ has finite index in $\Out_{H}(G)$. \label{cond:H_rigid}
	\end{myenum}
	\end{prop}
	\begin{proof}
		First, observe that $\ker f|_{\Aut_{\underline H}(G)} = \Inn(G) \cap \Aut_{\underline H}(G) \cong \Centr_G(H) / \Centr(G)$. Thus $\ker f|_{\Aut_{\underline H}(G)}$ is finite if and only if $\Centr(G)$ has finite index in $\Centr_G(H)$, i.e.~\ref{cond:finitecentralizer} holds.
		
		Next, recall that we have $\Out_{\underline H}(G) < \Out_{H}(G) = \Out_{[H]}(G) < \Out(G)$. Thus $\Out_{\underline H}(G)$ has finite index in $\Out(G)$ if and only if condition~\ref{cond:H_rigid} holds and $\Out_{[H]}(G)$ has finite index in $\Out(G)$, which is a restatement of~\ref{cond:preserve_conj_class}.
	\end{proof}
	
	
	\subsection{Self-normalized subgroups and the restriction homomorphism} Suppose that the group $H$ is self-normalized in $G$. We let $\lambda : \Aut_{H}(G) \to \Aut(H)$ be the natural restriction map. Remark that $\Aut_{\underline H}(G) = \ker \lambda$. Since $H$ is self-normalized in $G$, the map $\lambda$ induces a homomorphism $\varphi : \Out_{H}(G) \to \Out(H)$ whose kernel is exactly $\Out_{\underline H}(G)$. Thus, in this particular setting conditions~\ref{cond:finitecentralizer} and~\ref{cond:H_rigid} admit further restatements.
	
	\begin{prop} \label{prop:selfnorm} Let $G$ be a group, and $H$ be a self-normalized subgroup of $G$. Then $H$ is an \beep subgroup of $G$ if and only if the following three conditions hold:
	\begin{myenump}{\BP} \setcounter{enumi}{0}
		\item $\Centr(H)/\Centr(G)$ is finite. \label{cond:H_selfnorm_finitecenter}
	\end{myenump}
	\begin{myenum}{\BP} \setcounter{enumi}{1}
		\item A finite index subgroup of $\Aut(G)$ preserves the conjugacy class of $H$.
	\end{myenum}
	\begin{myenump}{\BP} \setcounter{enumi}{2}
		\item The image of $\varphi : \Out_{H}(G) \to \Out(H)$ is finite. \label{cond:H_selfnorm_rigid} \qed
	\end{myenump}
	\end{prop}
	
	\begin{remark} Even if $H$ is not self-normalized in $G$, condition~\ref{cond:H_rigid} holds as soon as $\Out(H)$ is finite. Indeed, letting $N = \Aut_H(G) \cap \Inn(G)$ we have 
	\begin{align*} [\Out_H(G):\Out_{\underline H}(G)] &= [\Aut_H(G):\langle \Aut_{\underline H}(G), N\rangle]
	\\ &= [\lambda(\Aut_H(G)) : \lambda(N)] 
	\\ &\leq [\Aut(H) : \Inn(H)] = |\Out(H)|
	\end{align*}
	\end{remark}
	
	We single out an important particular case of Proposition~\ref{prop:restatement}.
	\begin{corollary} \label{cor:selfnormgroup} Let $G$ be a group, and let $H$ be a subgroup of $G$ such that:
		\begin{enumerate} 
			\item there are finitely many conjugacy classes of subgroups of $G$ isomorphic to $H$.
			\item $N_G(H)$ is finite.
		\end{enumerate}
		Then $H$ is an \beep subgroup of $G$.
	\end{corollary}
	Condition (2) above forces $H$ to be finite. However, the conclusion remains true if condition (2) is replaced by requiring that $\Out(H)$ and $\Centr_G(H)/\Centr(G)$ are finite. Theorem~\ref{thm:main_sample} is a direct consequence Theorem~\ref{theorem:Baumslag}, Proposition~\ref{prop:strategy} and Corollary~\ref{cor:selfnormgroup}.
	
\subsection{\Beep subgroups from actions on $\CAT(0)$ spaces}
		
	Throughout this section, we let $X$ be a complete $\CAT(0)$ space and we let $G$ be a group acting by isometries on $X$. For $H < G$ we let $\Fix H$ be the set of fixed points $\{x \in X \mid hx=x \forall h \in H\}$. A subgroup $H< G$ is called \textbf{elliptic} if $\Fix H \neq \emptyset$.
	
	The following lemma provides a geometric criterion to find self-normalized subgroups of groups acting on $\CAT(0)$ spaces.
	\begin{lemma} \label{lemma:maxelliptic} Let $G$ be a group acting by isometries on a complete $\CAT(0)$ space $X$.
		\begin{enumerate} 
			\item Any elliptic subgroup $H_0$ with $\Fix H_0$ bounded is contained in a maximal elliptic subgroup $H$.
			\item If $H < G$ is a maximal elliptic subgroup with $\Fix H$ bounded then $\Norm_G(H) = H$.
		\end{enumerate}
	\end{lemma}
	\begin{proof}
		We show that the union $H=\cup_\alpha H_\alpha$ of a chain of elliptic subgroups $H_\alpha$ containing $H_0$ is elliptic. Indeed, each $\Fix H_\alpha$ is nonempty, closed, convex and bounded, and if $H_\alpha < H_\beta$ then $\Fix H_\alpha \supset \Fix H_\beta$. Since any finite subfamily of $\{\Fix H_\alpha\}_\alpha$ has nonempty intersection, it follows from \cite[Theorem 14]{Monod06} that $\Fix H = \cap_\alpha \Fix H_\alpha$ is nonempty. The first assertion now follows from Zorn's Lemma.
		
		In order to prove the second assertion, we observe that $\Norm_G(H)$ stabilizes the bounded set $\Fix H$, so that $\Norm_G(H)$ fixes a point, namely the center of $\Fix H$ \cite[Prop. II.2.7]{BridHaef99}. It now follows from the maximality assumption that $\Norm_G(H)=H$.
	\end{proof}
	
	We thus get the following geometric version of Proposition~\ref{prop:selfnorm}.
	\begin{prop} \label{prop:CAT(0)}Let $G$ be a group acting by isometries on a complete $\CAT(0)$ space $X$, and let $H$ be a maximal elliptic subgroup with $\Fix H$ bounded. Suppose that:
	\begin{enumerate}
		\item $\Centr(H)/\Centr(G)$ is finite;
		\item $\Aut(G)$ preserves the family of elliptic subgroups of $G$;
		\item there are finitely many conjugacy classes of maximal elliptic subgroups;
		\item the image of $\varphi : \Out_H(G) \to \Out(H)$ is finite.
	\end{enumerate}
	Then $H$ is an \beep subgroup of $G$.
	\end{prop}
	\begin{proof}
		By Lemma~\ref{lemma:maxelliptic} the subgroup $H$ is self-normalized in $G$. Thus by Proposition~\ref{prop:selfnorm} we only need to check condition~\ref{cond:preserve_conj_class}. Since $\Aut(G)$ preserves elliptic subgroups, it also preserves maximal elliptic subgroups. By hypothesis there are only finitely many conjugacy classes of maximal elliptic subgroups, so that a finite index subgroup of $\Aut(G)$ preserves the conjugacy class of $H$.
	\end{proof}
	
	Restricting to the case of properly discontinuous and cocompact actions we get
	\begin{corollary} \label{cor:geometricCAT0} Let $G$ be a group acting properly discontinuously and cocompactly by isometries on a complete $\CAT(0)$ space $X$, and let $H_0$ be a finite subgroup whose set of fixed points is bounded. Then any maximal finite subgroup $H$ containing $H_0$ is an \beep subgroup.
	\end{corollary}
	\begin{proof} Since $G$ acts properly discontinuously and cocompactly on $X$, elliptic subgroups coincide with finite subgroups, and there are only finitely many conjugacy classes of finite subgroup. In particular there exists a maximal finite finite subgroup $H$ containing $H_0$. It is now straightforward to check all hypotheses of Proposition~\ref{prop:CAT(0)}.
	\end{proof}
	Theorem~\ref{thm:geometricsplit} follows from Proposition~\ref{prop:strategy} and Corollary~\ref{cor:geometricCAT0}.
	
	\begin{remark} 
		If a complete $\CAT(0)$ space $X$ is proper, then a subset $Y \subset X$ is unbounded if and only if it contains a geodesic ray. In that case an elliptic subgroup $H$ has a bounded set of fixed point if and only if $H$ does not fix a point of the visual boundary $\partial X$.
	
		If a group $G$ acts properly discontinuously and cocompactly on a complete $\CAT(0)$ space, then $X$ is proper and the normalizer of a finite subgroup $H$ acts cocompactly on $\Fix H$ (see Remark 2 following Theorem 3.2 in \cite{Ruane01}). Thus for $H$ a finite subgroup of $G$ the following are equivalent:
	\begin{enumerate}
		\item $\Norm_G(H)$ is finite
		\item $\Fix H$ is bounded
		\item $H$ fixes no point of the $\CAT(0)$ boundary $\partial X$.
	\end{enumerate}
	\end{remark}
	
	\section{Applications} \label{sec:applications}
	
	We now turn to applications of the results of the preceding section. Recall that if the map $\Aut(G) \to \Out(G)$ splits virtually for a finitely generated residually finite group $G$, then necessarily $\Out(G)$ is residually finite. Thus the results below also provide an elementary strategy to show that $\Out(G)$ is residually finite for many groups. This is the only strategy currently known that applies to all finitely generated Coxeter groups (see Section~\ref{sec:residualfiniteness} for other strategies).
	
	\subsection{Coxeter groups} We prove Theorem~\ref{thm:Coxeter} from the Introduction. Let $G$ be a finitely generated Coxeter group. Since $G$ is $\RR$-linear (through the geometric representation \cite[Appendix D]{Davis08}) and finitely generated, it follows from Malcev's theorem that $G$ is residually finite. We refer the reader unfamiliar with Coxeter groups and their Davis complex to the book \cite{Davis08}.
	
	Moreover, $G$ acts properly and cocompactly on its Davis complex $X$, which is a $\CAT(0)$ piecewise Euclidean cell complex \cite{Moussong88}. Maximal cells of this complex are in bijective correspondence with maximal finite subgroups of $G$, which are themselves (finite) Coxeter groups. More precisely, given a maximal finite group $H$ the Euclidean metric on the corresponding maximal cell $Y$ is modeled on the convex hull of an orbit of a faithful representation of $H$ as a Euclidean reflection group. Moreover, the only fixed point under this representation is the origin, i.e. the center of the cell. Since the only fixed point in $Y$ of $H$ lies in the interior of $Y$, and since fixed point sets are convex, it follows that $H$ fixes a unique point in $X$. We can therefore apply Theorem~\ref{thm:geometricsplit} to the maximal finite subgroup $H$, so that the map $\Aut(G) \to \Out(G)$ splits virtually.
	
	We now use the results recalled in Section~\ref{sec:virtually_residually_p}. As said above, $G$ is $\RR$-linear. Thus $G$ is virtually residually $p$ for all but finitely many primes $p$ by Platonov's theorem. Hence the same holds for $\Aut(G)$ by Lubotzky's result, and thus for $\Out(G)$ as well by the above splitting. Finally, if a group $G$ is virtually residually $p$ for some prime $p$ then it is obviously residually finite, and if moreover $G$ is virtually residually $q$ for some distinct prime $q$ then $G$ is virtually torsion-free by Remark~\ref{rmk:torsion-free}. This finishes the proof of Theorem~\ref{thm:Coxeter}.
	
	\begin{remark}[Universal Coxeter groups] Let $W_n = \langle x_1, \ldots, x_n \mid x_i^2 = 1\rangle \cong \ast_{i = 1}^n \ZZ/2\ZZ$ be the universal Coxeter group of rank $n$. Let $p : W_{n+1} \twoheadrightarrow W_n \cong W_{n+1} / \langle\!\langle x_{n+1} \rangle\!\rangle$ be the natural projection and $\iota : W_n \hookrightarrow W_{n+1}$ be the natural section of $p$. The subgroup of $\Aut(W_{n+1})$ which preserves the conjugacy class of $x_{n+1}$ has finite index. Thus the map $p$ induces a virtual map $p^\ast :  \Aut(W_{n+1}) \twoheadrightarrow \Aut(W_n)$ and $\iota$ induces a section $\iota^\ast : \Aut(W_n) \hookrightarrow \Aut(W_{n+1})$ of $p^\ast$. Now since $p^\ast (\Inn(W_{n+1}) = \Inn(W_n)$ it follows that $p^\ast$ induces a virtual map $\hat p : \Out(W_{n+1}) \twoheadrightarrow \Out(W_n)$. On the other hand is is not obvious that $\hat p$ has a virtual section since $\iota^\ast(\Inn(W_n)) \cap \Inn(W_{n+1}) = \{1\}$. However, if we let $\sigma_n$ be a virtual section of the map $\Aut(W_n) \to \Out(W_n)$ given by Theorem~\ref{thm:Coxeter} then one can produce a virtual map $\hat \iota$ so that the diagram of virtual maps in Figure~\ref{fig:section} commutes. In particular we recover the known fact that $\Out(W_n)$ virtually embeds in $\Out(W_{n+1})$.
	\end{remark}
	\begin{figure}[htb]
			\[ \begin{xy}
		  \xymatrix{
		     \Aut(W_n) \ar@{^{(}->}@<1ex>[r]^{\iota^\ast} \ar@{->>}[d] \ar@{}[dr] |{\text{virt.}} & \Aut(W_{n+1}) \ar@{->>}[l]^{p^\ast} \ar@{->>}[d]  \\
		     \Out(W_n) \ar@{^{(}-->}[r]^{\hat \iota} \ar@{^{(}->}@<1ex>[u]^{\sigma_n} & \Out(W_{n+1}) \ar@{->>}@<1ex>[l]^{\hat p} \ar@{_{(}->}@<-1ex>[u]_{\sigma_{n+1}}  
		  	}
			\end{xy}\]
			\caption{The diagram of virtual maps defining the virtual section $\hat \iota$ of $\hat p$.}
			\label{fig:section}
		\end{figure}
	
	\subsection{Fuchsian groups with torsion} Let $X$ be a closed hyperbolic $2$-orbifold with at least one cone point of order $n \geq 2$ or with at least one reflector corner (i.e. $X$ is not a closed surface nor a compact surface with each boundary component consisting of a single mirror), and let $G$ be its orbifold fundamental group. In other words, $G$ acts faithfully, properly discontinuously and cocompactly by isometries on the hyperbolic plane, and $G$ contains a rotation of angle $2\pi/n$ for some $n \geq 2$ (either coming from a cone point of order $n$ or from a reflector corner of angle $\pi/n$).
	
	The group $G$ is finitely generated and linear (indeed $G$ is a uniform lattice in $\PGL_2 \RR$). Hence by Malcev's theorem $G$ is residually finite. Thus Theorem~\ref{thm:geometricsplit} applies to the subgroup $H$ generated by the rotation, so that the map $\Aut(G) \to \Out(G)$ splits virtually.
	
	
	\subsection{Free products} Let $G = G_1 * H$ where $G$ is finitely generated and $H \neq \{1\}$ is freely indecomposable with finite center and such that $\Out(H)$ is finite (e.g. $H$ is finite). Since $G$ is finitely generated, it admits a Grushko decomposition as a free product of finitely many freely indecomposable groups $\{H_i\}_{1 \leq i \leq n}$ (one of which is conjugate to $H$) and a free group of finite rank $F_k$. We let $G$ act on $T$ a Bass-Serre tree corresponding to this free product decomposition, i.e. a tree on which $G$ acts cocompactly with trivial edge stabilizers and with each vertex stabilizer conjugate to some $H_i$. Since the conjugates of $H_i$ are the maximal freely indecomposable subgroups of $G$ (not isomorphic to $\ZZ$), it follows that $\Aut(G)$ preserves the family of elliptic subgroups. Finally, since edge stabilizers are trivial, it follows that $\Fix H$ consists of a single vertex of $T$. Thus all conditions of Proposition~\ref{prop:CAT(0)} are satisfied for the action of $G$ on $T$, so that $H$ is an \beep subgroup.
	
	Combining this with Proposition~\ref{prop:strategy} or Remark~\ref{remark:trivialcenter} we obtain
	\begin{corollary} Let $G = G_1 * H$ where $G$ is finitely generated and $H \neq \{1\}$ is freely indecomposable with $\Out(H)$ finite. Suppose moreover that one of the following holds
		\begin{itemize}
			\item either $G$ is residually finite and $H$ has finite center
			\item or $H$ has trivial center.
		\end{itemize}
		Then the map $\Aut(G) \to \Out(G)$ splits virtually.
	\end{corollary}	
	
	
	\subsection{One-ended hyperbolic groups} Here, we use a similar argument as in the previous section but for the action of a one-ended hyperbolic group on its JSJ tree (see \cite{Sela97,Bowditch98}).
		\begin{corollary} Let $G$ be a one-ended hyperbolic group with a JSJ decomposition admitting a non-elementary rigid vertex. Suppose moreover that $G$ is either residually finite or torsion-free. Then the map $\Aut(G) \to \Out(G)$ splits virtually.
		\end{corollary}
		\begin{proof} Let $H$ be a non-elementary vertex group in the JSJ decomposition of $G$, and let $T$ be the corresponding Bass-Serre tree. Most properties we need are features of the JSJ decomposition:
			\begin{itemize}
				\item The action of $G$ on $T$ is cocompact, i.e. the quotient graph of groups has finitely many vertices and edges. In particular, there are finitely many conjugacy classes of maximal elliptic subgroups.
				\item $\Aut(G)$ preserves elliptic subgroups.
				\item Edge stabilizers are virtually cyclic, so that $\Fix H$ consists of a single vertex of $T$.
				\item The map $\varphi : \Out_H(G) \to \Out(H)$ is finite.
			\end{itemize}
			Note moreover that $H$ is a non-elementary hyperbolic group and hence has finite center (and thus trivial center if $G$ is torsion-free). Thus Proposition~\ref{prop:CAT(0)} applies and $H$ is an \beep subgroup. The result now follows from Remark~\ref{remark:trivialcenter} in the torsion-free case, or from Theorem~\ref{theorem:Baumslag} and Proposition~\ref{prop:strategy} in the residually finite case. 
		\end{proof}
		
	\subsection{One-relator groups with torsion} We show that Theorem~\ref{thm:main_sample} applies to all one-relator groups with torsion.
	\begin{corollary} Let $G$ be a one-relator group with torsion, i.e. $G = \langle x_1, \ldots, x_k \mid R^n\rangle$ for $n > 1$. Then the map $\Aut(G) \to \Out(G)$ splits virtually.
	\end{corollary} 
	\begin{proof} Let $H = \langle R \rangle$. In order to apply Theorem~\ref{thm:main_sample} we collect known results:
	\begin{itemize}
		\item Any finite subgroup of $G$ is conjugate to a subgroup of $H$, so that $G$ has finitely many conjugacy classes of finite subgroups.
		\item Newman's Spelling Theorem \cite{Newman68} implies that $\Norm_G(H) = H$.
		\item It follows from the recent body of work of Wise and collaborators, announced in \cite{Wise09} that one-relator groups with torsion are linear over $\mathbf Z$ and in particular residually finite. \qedhere
	\end{itemize}
	\end{proof}

	

	\bibliographystyle{amsalpha}
	\bibliography{biblio}

\end{document}